\documentclass{article}
\usepackage{amssymb,amsfonts,amsmath,amsthm,subfigure}
\usepackage{authblk}
\usepackage{txfonts}
\usepackage{graphicx}
\def\rightangle{\vcenter{\hsize5.5pt
    \hbox to5.5pt{\vrule height7pt\hfill}
    \hrule}}
\def\rtangle{\mathrel{\rightangle}}
\numberwithin{equation}{section} 

\newtheorem{thm}{Theorem}[section]
\newtheorem{defn}[thm]{Definition}
\newtheorem{cor}[thm]{Corollary}
\newtheorem{lemma}[thm]{Lemma}

\newtheorem{rmrk}[thm]{Remark}

\newcommand{\e}{\varepsilon}
\newcommand{\R}{\mathbb{R}}

\newtheorem{prop}[thm]{Proposition}

\newcommand{\abs}[1]{\left\vert{#1}\right\vert}

\newcommand{\ba}{\begin{array}}
\newcommand{\ea}{\end{array}}

\newcommand{\bthm}{\begin{thm}}
\newcommand{\ethm}{\end{thm}}
\newcommand{\bstp}{\begin{stp}}
\newcommand{\estp}{\end{stp}}
\newcommand{\blemma}{\begin{lemma}}
\newcommand{\elemma}{\end{lemma}}
\newcommand{\bprop}{\begin{prop}}
\newcommand{\eprop}{\end{prop}}
\newcommand{\bpf}{\begin{pf}}
\newcommand{\epf}{\end{pf}}
\newcommand{\bdefn}{\begin{defn}}
\newcommand{\dH}{d\mathcal{H}}
\newcommand{\edefn}{\end{defn}}
\newcommand{\brk}{\begin{rmrk}}
\newcommand{\erk}{\end{rmrk}}
\newcommand{\bcrl}{\begin{crl}}
\newcommand{\ecrl}{\end{crl}}
\newcommand{\norm}[1]{\left\|#1\right\|}

\newcommand{\beqn}{\begin{equation}}
\newcommand{\eeqn}{\end{equation}}

\renewcommand{\leq}{\leqslant}
\renewcommand{\geq}{\geqslant}

\newcommand{\beq}{\begin{equation}}
\newcommand{\eeq}{\end{equation}}
\newcommand{\bea}{\begin{eqnarray}}
\newcommand{\eea}{\end{eqnarray}}
\begin{document}
\renewcommand\Authfont{\small}
\renewcommand\Affilfont{\itshape\footnotesize}
\newcommand{\Et}{\tilde{E}_\e^{\gamma}}
\newcommand{\E}{E_{\Omega_\e}^{\gamma}}
\newcommand{\Ez}{E_0^{\gamma}}
\newcommand{\Oe}{\Omega_\e}
\title{Cascade of minimizers for a nonlocal isoperimetric problem in thin domains}

\author[1]{Massimiliano Morini\footnote{massimiliano.morini@unipr.it}}
\author[2]{Peter Sternberg\footnote{sternber@indiana.edu}}
\affil[1]{Department of Mathematics, University of Parma, Parma, Italy}
\affil[2]{Department of Mathematics, Indiana University, Bloomington, IN 47405}

\maketitle
\noindent{\bf Mathematics Subject Classification:} 49J45, 49Q20\\
\noindent{\bf Keywords:} nonlocal isoperimetric, global minimizers
\begin{abstract}
For $\Omega_\e=(0,\e)\times (0,1)$ a thin rectangle, we consider minimization of
the two-dimensional nonlocal isoperimetric
problem given by 
\[
\inf_u E^{\gamma}_{\Omega_\e}(u)\]
where
\[
E^{\gamma}_{\Omega_\e}(u):= P_{\Omega_\e}(\{u(x)=1\})+\gamma\int_{\Omega_\e}\abs{\nabla{v}}^2\,dx
\]
and the minimization is taken over competitors $u\in BV(\Omega_\e;\{\pm 1\})$ satisfying a mass constraint $\fint_{\Omega_\e}u=m$ for
some $m\in (-1,1)$. Here $P_{\Omega_\e}(\{u(x)=1\})$ denotes the perimeter of the set $\{u(x)=1\}$ in $\Omega_\e$, $\fint$ denotes the integral average and $v$ denotes the solution
to the Poisson problem
\[
-\Delta v=u-m\;\mbox{in}\;\Omega_\e,\quad\nabla v\cdot  n_{\partial\Omega_\e}=0\;\mbox{on}\;\partial\Omega_\e,\quad\int_{\Omega_\e}v=0.\]
We show that a striped pattern is the minimizer for $\e\ll 1$ with the number of stripes growing like $\gamma^{1/3}$ as $\gamma\to\infty.$ In the process, we show that
stable lamellar patterns are in fact $L^1$ local minimizers in rectangular domains.
We then present generalizations of this result to higher dimensions.
\end{abstract}

\section{Introduction}
In nonlocal isoperimetric problems currently of interest, one considers a perturbation of
the classical isoperimetric problem by a term that favors high oscillation.
This tension between terms favoring low and high surface area respectively leads
to a rich and not well understood energy landscape. To date, identification of minimizers
has been largely limited to parameter regimes in which the perimeter term dominates and so
it is the purpose of this article to present a setting, namely thin domains, that allows for such an identification at all magnitudes
of the nonlocal perturbation.

To state our problem precisely, given a bounded domain $\Omega\subset\R^n$ and a number $\gamma\geq 0$ we consider the minimization
of the functional
\beq
E^{\gamma}_{\Omega}(u):= P_{\Omega}(\{u(x)=1\})+\gamma\int_{\Omega}\abs{\nabla{v}}^2\,dx\label{Egammaintro}
\eeq
over the set of competitors $u\in BV(\Omega;\{\pm 1\})$ satisfying the mass constraint $\fint_{\Omega}u=m$ for
some $m\in (-1,1)$. Here $P_{\Omega}(\{u(x)=1\})$ denotes the perimeter of the set $\{u(x)=1\}$ in $\Omega$, $\fint$ denotes the integral average and $v$ denotes the solution
to the Poisson problem
\begin{equation}
-\Delta v=u-m\;\mbox{in}\;\Omega,\quad\nabla v\cdot  n_{\partial\Omega}=0\;\mbox{on}\;\partial\Omega,\quad\int_{\Omega}v=0,\label{Poisson}\end{equation}
with $n_{\partial\Omega}$ denoting the outer unit normal to $\partial\Omega$. We recall that $P_{\Omega}(\{u(x)=1\})$ can alternatively be expressed as
$\frac{1}{2}\abs{\nabla u}(\Omega)$ where $\abs{\nabla u}(\Omega)$ denotes the total variation of the vector-valued measure $\nabla u$, cf. \cite{Gu}.

The functional $E^{\gamma}_{\Omega}$ arises as the sharp interface $\Gamma$-limit as $\delta\to 0$ of the Ohta-Kawasaki functional modeling phase separation
in diblock co-polymers
\[
u\mapsto \int_{\Omega}\frac{1}{\delta}(u^2-1)^2+\delta \abs{\nabla u}^2+\gamma\int_{\Omega}\abs{\nabla{v}}^2\,dx,\]
see e.g. \cite{CR,OK,RW1}. Thus, at least on a qualitative level, one expects that minimizers of \eqref{Egammaintro} should bear some resemblance
to the pictures of phase separation reported experimentally in the co-polymer literature, e.g. in \cite{AFH,BF,TAHH}. The most striking feature of these images
in parameter regimes where the nonlocality dominates is the emergence of small periodically arrayed cells inside of which the interface $\partial\{u=1\}$ resembles a
constant mean curvature surface.

Now as we show in Section 3, and as was already studied earlier in e.g. \cite{RT} and \cite{RW1}, in one dimension when $\Omega$ is simply an interval, the problem can be explicitly solved. Here it is easy to see that
minimizers are essentially periodic--up to adjustments at the boundary to accommodate the Neumann boundary conditions--with oscillations on the order of $\gamma^{1/3}$ in the regime $\gamma\gg 1$, a scaling that has previously been noted for example in \cite{CM}. (A similar conclusion for one-dimensional minimizers of
the Ohta-Kawasaki functional can also be drawn but this is nontrivial, see \cite{Mu}.)
When $n\geq 2$ however, the problem becomes quite subtle. To date, the only general result in this direction is that of \cite{ACO} where the authors
show, roughly speaking, that energy tends to distribute uniformly in two dimensions. A corresponding result for Ohta-Kawasaki was obtained more recently in \cite{Sp}.

With regard to characterizing more precisely the global minimizer,
progress up to now has been largely limited to parameter regimes where perimeter dominates. When $\gamma$
is small this includes \cite{ST,T}. There is also a growing literature on asymptotic regimes where $m$ is near $1$ or $-1$, on the setting $\Omega=\R^n$
and on related perturbations of the isoperimetric problem, some of which arise as $\Gamma$-limits of Ohta-Kawasaki under different scalings, see for example 
\cite{BC,CP,CiS,GP,GMS,J1,KM,KM1,M1,PV}. In a different vein, the existence of increasingly intricate critical points and local minimizers for \eqref{Egammaintro} and related nonlocal sharp interface problems has been one thrust of the research program of Ren and Wei, see for example \cite{RW1,RW2,RW3} and the references therein.

In this article we investigate a multi-dimensional setting where we can identify the global minimizer of \eqref{Egammaintro} for {\it all} values of $\gamma$. 
The simplest such example is the case where $\Omega$ is a thin rectangle given by $\Oe:=(0,\e)\times (0,1)$ for $\e$ small. Our main result here, Theorem \ref{main}, states that
for any value of $\gamma$, when $\e$ is sufficiently small, the global minimizer of $E^{\gamma}_{\Oe}$ coincides with the minimizer of the one-dimensional problem
posed on the unit interval. Since the one-dimensional problem is minimized by a piecewise constant function with more and more jumps in the regime $\gamma\gg 1$,
this implies that as $\gamma$ grows the minimizer of the two-dimensional problem exhibits a cascade of oscillations through a pattern of more and more horizontal stripes. The relationship
between the number of stripes $k$ and the value of $\gamma$ is given explicitly in \eqref{kwins}. We then
apply the technique to cover domains of the form $(0,\e)^{\ell}\times(0,1)$ for any positive integer $\ell$ and more general thin domains in Theorems \ref{main2} and \ref{main3}.

Let us describe the main ingredients in 
the method of proof for the main result, Theorem \ref{main}. A first step is the establishing of $\Gamma$-convergence of \eqref{Egammaintro} to a one-dimensional energy in the setting where $\Omega=\Oe$ as $\e\to 0.$ This is accomplished in Section 2. Section 3 contains the explicit identification of the global minimizer of the one-dimensional $\Gamma$-limit alluded to earlier.
In Section 4 we give a proof of the two-dimensional stability of the one-dimensional minimizers. This stability was first addressed in the periodic setting in \cite{CM}, in which a more general machinery was introduced for studying stability of critical points in a variety of regimes, including higher dimensions. 
Through reflection, this yields stability for the Neumann problem of our setting. However, we include our proof here both for the stake of self-containment and because our argument is completely different from the earlier one and we find it to be quite a bit simpler. In Section 5 we first establish the appropriate modifications of  
the stability $\implies$ local minimality results of \cite{AFM} and \cite{Julin} to this setting of Neumann boundary conditions in domains with corner singularities. This in particular yields the new result that stable lamellar patterns are in fact $L^1$ local minimizers.
 Then we synthesize all of these tools to prove the global minimality in two-dimensional thin rectangles of the one-dimensional (lamellar) patterns.  Finally Section 6 contains a few generalizations
to thin domains in arbitrary dimensions.

\section{$\Gamma$-convergence to the $1d$ nonlocal isoperimetric problem }

For $\e>0$ we let $\Omega_\e$ denote the rectangle $(0,\e)\times(0,1)$. Then
for $\gamma>0$ and any $m\in(-1,1)$ we introduce the functional $\E:L^1(\Oe)\to\R$ given by \eqref{Egammaintro}
with $\Omega$ replaced by $\Omega_\e$.
We wish to identify the $\Gamma$-limit of $\E$ as $\e\to 0$
and to this end, given any $u\in BV(\Oe;\{\pm 1\})$, we
denote by $\tilde{u}:\Omega_1\to\R$ the function satisfying $\tilde{u}(x_1,y_1)=u(\e x_1,y_1)$ and readily compute that
\[
P_{\Oe}(\{u(x)=1\})=\e\int_{\partial^*\{\tilde{u}=1\}}\sqrt{\frac{1}{\e^2}n_1^2+n_2^2}\,\dH^1\]
where $(n_1,n_2)$ is the outer normal to the reduced boundary $\partial^*\{\tilde{u}=1\}$ and the integration is with respect
to one-dimensional Hausdorff measure, cf. \cite{Gu}. Similarly if $v=v(x,y)$ is the solution to \eqref{Poisson} associated with $u$, then the
function $\tilde{v}(x_1,y_1):=v(\e x_1,y_1)$ satisfies
\beq
-\frac{1}{\e^2}\tilde{v}_{x_1x_1}-\tilde{v}_{y_1y_1}=\tilde{u}-m\label{epn}\eeq
along with homogeneous Neumann boundary conditions on $\partial\Omega_1$ and zero mean.
Consequently, we have
\[
\int_{\Oe}\abs{\nabla v}^2\,dx\,dy=\e\int_{\Omega_1}\left(\frac{1}{\e^2}\tilde{v}_{x_1}^2+\tilde{v}_{y_1}^2\right)\,dx_1\,dy_1\]
and so
\begin{eqnarray}&&\frac{1}{\e}\E(u)=\tilde{E}_{\e}^{\gamma}(\tilde{u}):=\nonumber\\
&&\left\{\begin{matrix} \int_{\partial^*\{\tilde{u}=1\}\cap\Omega_1}\sqrt{\frac{1}{\e^2}n_1^2+n_2^2}\,\dH^1
+\gamma\int_{\Omega_1}\left(\frac{1}{\e^2}\tilde{v}_{x_1}^2+\tilde{v}_{y_1}^2\right)\,dx_1\,dy_1&\mbox{if}\;\tilde{u}\in BV(\Omega_1;\{\pm 1\}),\;\fint_{\Omega_1}\tilde{u}=m\\
+\infty&\quad\mbox{otherwise,}\end{matrix}\right.\nonumber\\
&&\label{tEdefn}
\end{eqnarray}
We then establish $\Gamma$-convergence of  $\tilde{E}_{\e}^{\gamma}$ to the energy corresponding to the 1d nonlocal isoperimetric problem.
\begin{thm}\label{gc}
As $\e\to 0$, the functionals $\tilde{E}_{\e}^{\gamma}$ $\Gamma$-converge in $L^1(\Omega_1)$ to $E_0^{\gamma}$ given by
\[
E_0^{\gamma}(\tilde{u}):=\left\{\begin{matrix} \frac{1}{2}\abs{\tilde{u}_{y_1}}(\Omega_1)
+\gamma\int_0^1\tilde{v}_{y_1}^2\,dy_1&\mbox{if}\;\abs{\tilde{u}_{x_1}}(\Omega_1)=0,\;\tilde{u}\in BV(\Omega_1;\{\pm 1\}),\;\fint_{\Omega_1}\tilde{u}=m\\
+\infty&\quad\mbox{otherwise,}\end{matrix}\right.
\]
where $\abs{\tilde{u}_{x_1}}$ and $\abs{\tilde{u}_{y_1}}$ denote the total variation
of the measures $\tilde{u}_{x_1}$ and $\tilde{u}_{y_1}$ and $\tilde{v}=\tilde{v}(y_1)$ solves
\beq
-\tilde{v}_{y_1y_1}=\tilde{u}-m\;\mbox{for}\;0<y_1<1,\quad \tilde{v}_{y_1}(0)=0=\tilde{v}_{y_1}(1).\label{1dp}\eeq
\end{thm}
\begin{proof}
Given $\tilde{u}\in L^1(\Omega_1)$ let us first assume that $\tilde{u}_\e\to \tilde{u}\;\mbox{in}\;L^1(\Omega_1)$. Then we will argue that
\begin{equation} \liminf_{\e\to 0}\tilde{E}_{\e}^{\gamma}(\tilde{u}_\e)\geq
E_0^{\gamma}(\tilde{u}).\label{lsc}\end{equation}
Clearly we may assume
\beq
\liminf_{\e\to 0}\tilde{E}_{\e}^{\gamma}(\tilde{u}_\e)<\infty,
\label{upbd}
\eeq
and in particular that $\{\tilde{u}_\e\}\in BV(\Omega;\{\pm 1\})$
so we may write
\[
\tilde{u}_\e=\left\{\begin{matrix} 1&\mbox{in}\;A_\e\\-1&\mbox{in}\;\Omega_1\setminus A_\e,\\
\end{matrix}\right.\qquad
\tilde{u}=\left\{\begin{matrix} 1&\mbox{in}\;A\\-1&\mbox{in}\;\Omega_1\setminus A,\\
\end{matrix}\right.
\]
for sets of finite perimeter $A_\e$ and $A$, in light of the lower-semicontinuity of the total variation under
$L^1$-convergence.

Now if $\abs{\tilde{u}_{x_1}}(\Omega_1)>0$ then we find
\[
\liminf_{\e\to 0}\tilde{E}_{\e}^{\gamma}(\tilde{u}_\e)\geq \liminf_{\e\to 0}\frac{1}{\e}\int_{\partial^* A_\e}
\abs{n_1}\dH^1=\frac{1}{2} \liminf_{\e\to 0}\frac{1}{\e}\abs{(\tilde{u}_\e)_{x_1}}(\Omega_1)=\infty
\]
since $\frac{1}{2} \liminf_{\e\to 0}\abs{(\tilde{u}_\e)_{x_1}}(\Omega_1)\geq\frac{1}{2}\abs{\tilde{u}_{x_1}}(\Omega_1)>0$.
Hence we may assume $\abs{\tilde{u}_{x_1}}(\Omega_1)=0$. In turn, this implies 
  that, up to choosing the right Lebesgue representative, $\tilde{u}=\tilde{u}(y_1)$. Although this last point is standard, we write here the simple argument for the reader's convenience. Let $\tilde u_\delta:=\tilde u*\rho_\delta$, where $\rho_\delta$ denotes the standard mollifier. Note that $\tilde u_\delta$ is well defined on  $\Omega_1^\delta:=(\delta, 1-\delta)\times(\delta, 1-\delta)$ and 
  $(\tilde u_\delta)_{x_1}=\tilde u*(\rho_\delta)_{x_1}=0$ on $\Omega_1^\delta$. Thus, in particular, $\tilde u_\delta=\tilde u_\delta (y_1)$ on 
  $\Omega_1^\delta$. The conclusion follows by recalling that $\tilde u_\delta\to \tilde u$ a.e. in $\Omega_1$.
  
 Consequently, since $\chi_{A_\e}\to \chi_{A}$ in
$L^1(\Omega_1)$ we have
\[
\liminf_{\e\to 0}\int_{\partial^*\{\tilde{u}=1\}}\sqrt{\frac{1}{\e^2}n_1^2+n_2^2}\,\dH^1\geq
\liminf_{\e\to 0} P_{\Omega_1}(A_\e)\geq P_{\Omega_1}(A)=\frac{1}{2}\abs{\tilde{u}_{y_1}}(\Omega_1).\]
Turning to the lower-semi-continuity of the second integral in the definition of $\tilde{E}_{\e}^{\gamma}$
we note that \eqref{upbd} implies the uniform bound
\[
\int_{\Omega_1}\left(\frac{1}{\e^2}\tilde{v}_{x_1}^2+\tilde{v}_{y_1}^2\right)\,dx_1\,dy_1<C.
\]

In light of the Poincar\'e inequality for functions of zero mean, this leads to a uniform $H^1$ bound and yields
 the existence of a function $\hat{v}\in H^1(\Omega_1)$ with $\hat{v}=\hat{v}(y_1)$ such
that after passing to a subsequence (with subsequential notation suppressed), one has \beq
(\tilde{v}_\e)_{x_1}\to 0\quad\mbox{in}\;L^2(\Omega_1)\quad\mbox{and}\quad\tilde{v}_\e\rightharpoonup \hat{v}\quad
\mbox{in}\;H^1(\Omega_1).\label{wkh1}
\eeq
Hence, we have
\[
\liminf_{\e\to 0}\int_{\Omega_1}\left(\frac{1}{\e^2}(\tilde{v}_\e)_{x_1}^2+(\tilde{v}_\e)_{y_1}^2\right)\,dx_1\,dy_1\geq
\liminf_{\e\to 0}\int_{\Omega_1}(\tilde{v}_\e)_{y_1}^2\,dx_1\,dy_1\geq\int_{\Omega_1}\hat{v}_{y_1}^2\,dx_1\,dy_1=
\int_0^1\hat{v}_{y_1}^2\,dy_1.
\]
It remains to identify $\hat{v}$ with the solution $\tilde{v}$ to \eqref{1dp}. To this end we
consider the weak formulation of the PDE in \eqref{epn} subject to homogeneous Neumann boundary conditions, namely,
\[
\int_{\Omega_1}\frac{1}{\e^2}\phi_{x_1}(\tilde{v}_\e)_{x_1}+\phi_{y_1}(\tilde{v}_\e)_{y_1}\,dx_1\,dy_1=\int_{\Omega_1}\phi(\tilde{u}_\e-m)\,dx_1\,dy_1
\]
for any smooth function $\phi$ defined on $\overline{\Omega}_1.$ Making the choice of an arbitrary smooth $\phi$ depending only on $y_1$
we obtain
\[
\int_{\Omega_1}\phi_{y_1}(\tilde{v}_\e)_{y_1}\,dx_1\,dy_1=\int_{\Omega_1}\phi(\tilde{u}_\e-m)\,dx_1\,dy_1.
\]
We then pass to the limit using \eqref{wkh1} and the $L^1$ convergence of $\tilde{u}_\e$ to $\tilde{u}$ to find that $\hat{v}$ weakly solves the ODE
and boundary conditions of  \eqref{1dp}, hence $\hat{v}=\tilde{v}.$

The second requirement of $\Gamma$-convergence, namely the construction of a recovery sequence, say $w_\e\to \tilde{u}$ in $L^1(\Omega_1)$
such that $\tilde{E}_{\e}^{\gamma}(\tilde{w}_\e)\to E_0^{\gamma}(\tilde{u})$ is trivial as one simply takes $w_\e\equiv \tilde{u}$ for all
$\e$.
\end{proof}

Finally we note that $L^1(\Omega_1)$-compactness of energy bounded sequences follows immediately since the condition
$\sup_{\e}\tilde{E}_{\e}^{\gamma}(\tilde{u}_{\e})<\infty$ implies in particular a uniform BV bound on such a sequence $\tilde{u}_{\e}$.
In what follows we will use only the most basic property of $\Gamma$-convergence, namely that any limit of minimizers of $\tilde{E}_{\e}^{\gamma}$
is necessarily a minimizer of $E_0^{\gamma}.$

\section{Global minimizers of the $\Gamma$-limit}

Minimization of the one-dimensional energy $\Ez$ is a straight-forward exercise. For such an analysis, including a determination of local minimality of
$k$-jump critical points, one may look for example, in \cite[Proposition 3.3] {RW1}. 
For the sake of self-containment, however, and so as to express the results in our notation, we nonetheless present the explicit calculation in this
section over the parameter range $0\leq \gamma<\infty.$ We will
fix the mass constraint $m=0$ for convenience though similar calculations can be done for any value
of $m$ between $1$ and $-1$.

We recall that when posed in a general domain $\Omega$ in $n$-dimensional Euclidean space, a function $u\in BV(\Omega;\{\pm1\})$ is a {\em regular critical point} for the nonlocal isoperimetric problem
provided that $\partial \{u=1\}\cap \Omega$ is of class $C^2$ up to $\partial\Omega$  and
\begin{equation}
H(x)\,+\,4\gamma \,v(x)\,=const. \,\,\,\;\mbox{\rm for
all}\,\,\;x\in\partial \{u=1\}\cap\Omega,\label{2dcrit}
\end{equation}
along with an orthogonality condition along $\partial\Omega\cap \partial \{u=1\}$ (provided $\partial\Omega$ is smooth at such a point of intersection), where $H$ denotes the mean curvature of the free surface $\partial\{u=1\}\cap\Omega$, cf. e.g. \cite{CS} or \cite{CM}. For the one-dimensional problem $E_0^{\gamma}$, however, the criticality condition
reduces to simply
 \begin{equation}
v(x)\,=const. \,\,\,\;\mbox{\rm for
all}\,\,\;x\in\partial \{u=1\}\cap (0,1),\label{1dcrit}
\end{equation}
where
$v$ is the solution to the ODE
\beq
-v''=u\quad\mbox{on}\;0<y<1,\qquad v'(0)=0=v'(1).\label{1dode}
\eeq
We can naturally categorize the critical points in terms of the points in $(0,1)$ where $u$ jumps between $\pm 1$,
calling these points, say $\{y_j\}$, and then we note that \eqref{1dcrit} in particular implies that
$\int_{y_j}^{y_{j+1}}v'\,dy=0.$ From this condition and the mass constraint $\int_0^1 u\,dy=0$, one easily checks that, up to
multiplication by $-1$, there is a unique critical point having $k$ jumps, which we denote by $u_k$. Introducing
the notation
\beq
y_j=\frac{2j-1}{2k}\quad\mbox{for}\;j=1,2,\ldots,k,\label{yj}\eeq
 (which suppresses the dependence on $k$) we find that the critical point with $k$ jumps is given by
\beq
u_k(y):=\left\{\begin{matrix} 1&\mbox{for}\quad 0<y<y_1,\;y_2<y<y_3,\ldots,\;y_{k-1}<y<y_k\\
-1&\mbox{for}\quad y_1<y<y_2,\;y_3<y<y_4,\ldots,\;y_k<y<1\end{matrix}\right.\label{kodd}
\eeq
when $k$ is odd, and
\beq
u_k(y):=\left\{\begin{matrix} 1&\mbox{for}\quad 0<y<y_1,\;y_2<y<y_3,\ldots,\;y_k<y<1\\
-1&\mbox{for}\quad y_1<y<y_2,\;y_3<y<y_4,\ldots,\;y_{k-1}<y<y_k\end{matrix}\right.\label{keven}
\eeq
when $k$ is even. Then we denote by $v_k$ the corresponding solution to \eqref{1dode}. For example, in Figure \ref{vk} we depict
the solution in the case $k=5.$
\begin{figure}
\centerline{{\includegraphics[scale = 0.38, clip = true, trim = 5cm 1cm 4cm 2cm, angle = 90]{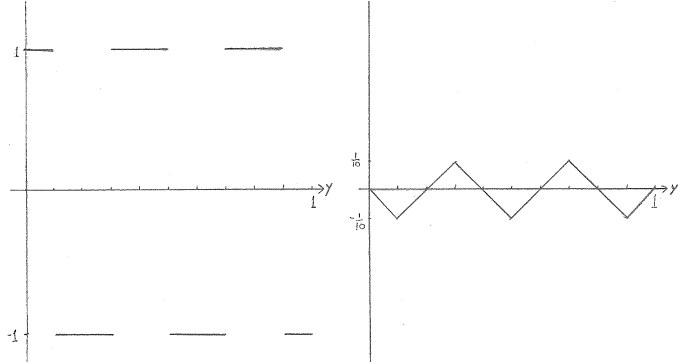}}} \caption{Graph of
the five jump critical point $u_5$ and the derivative of the corresponding solution $v_5$ to \eqref{1dode}.}
\label{vk}
\end{figure}
We then compute
\[
\int_0^1v_k'(y)^2\,dy=2k\int_0^\frac{1}{2k}y^2\,dy=\frac{1}{3}\bigg(\frac{1}{2k}\bigg)^2\]
so that
\beq
\Ez(u_k)=k+\frac{\gamma}{12k^2}.
\label{ukenergy}
\eeq
Fixing $\gamma>0$ and minimizing over $k$, we find that the minimizer of $\Ez$ will be given by $u_{k(\gamma)}$,
where $k(\gamma)$ is computable and will always be either the greatest integer less that $\bigg(\frac{\gamma}{6}\bigg)^{1/3}$
or the smallest integer bigger than $\bigg(\frac{\gamma}{6}\bigg)^{1/3}$. In particular, the number of interfaces of the minimizer
is a non-decreasing function of $\gamma$ that grows like $\gamma^{1/3}.$

 Alternatively, we observe that for any fixed integer $k$ the formula \eqref{ukenergy} is a linear function of $\gamma$
and the intersection point of any two of these lines corresponding to consecutive $k$ values moves monotonically to the right. This follows since
\[\Ez(u_{k-1})=\Ez(u_k)\quad\mbox{implies}\quad \gamma=\gamma_1(k):=\frac{12k^2(k-1)^2}{2k-1}\]
while  the condition
\[\Ez(u_{k})=\Ez(u_{k+1})\quad\mbox{implies}\quad \gamma=\gamma_2(k):=\frac{12k^2(k+1)^2}{2k+1}\]
and  one readily checks that $\gamma_1(k)<\gamma_2(k)$ for every positive integer $k$.
Thus,  
$u_k$ will be the minimizer for $\gamma$ lying in the interval $\gamma_1(k)\leq \gamma\leq \gamma_2(k)$. We therefore conclude:
\begin{prop}\label{global}
For a given positive integer $k$, the $k$ interface critical points $\pm u_k$ will be the global minimizers of $\Ez$ on the interval
\beq
\frac{12k^2(k-1)^2}{2k-1}<\gamma<\frac{12k^2(k+1)^2}{2k+1}\quad\mbox{for}\quad k=1,2,\ldots\label{kwins}
\eeq
\end{prop}

\section{Two-dimensional stability of the one-dimensional critical points}

Here we wish to determine the range of stability of the critical points $u_k$ defined in \eqref{kodd}-\eqref{keven}
with respect to the two-dimensional energy $\E.$ We refer the reader to \cite{CM} for an earlier derivation of stability of
lamellar patterns through an entirely different approach. By stability, we mean positivity of the second variation. We recall
that in a general domain $\Omega\subset\R^n$, $n$ arbitrary, the second variation of the nonlocal isoperimetric energy $E_{\Omega}^{\gamma}$
about a critical point $u\in BV(\Omega;\{\pm 1\})$
with $\Gamma:=\partial\{u=1\}\cap \Omega$ takes the form
\begin{multline}
 \delta^2 E_{\Omega}^{\gamma}(u;f)= \int_{\Gamma}\big(\abs{\nabla_{\Gamma}f}^{2}-\norm{B_{\Gamma}}^2 f^{2}\big)\,\dH^{n-1}\,
 -\int_{\overline \Gamma\cap\partial\Omega}B_{\partial\Omega}(n_{\Gamma},n_{\Gamma})f^2\, \dH^{n-2}\\
 +  8\gamma \int_{\Gamma}
\int_{\Gamma}G({\bf {x}},{\bf{\tilde{x}}})\,f({\bf{x}})\,f({\bf{\tilde{x}}})\,\dH^{n-1}({\bf{x}})\,\dH^{n-1}({\bf{\tilde{x}}})
+\; 4\gamma \int_{\Gamma}\nabla
v\cdot n_{\Gamma}\,f^{2}\,\dH^{n-1}.\label{gensecondvar}
\end{multline}
Here eligible functions $f$ are those lying in $H^1(\Gamma)$ and satisfying $\int_{\Gamma} f\,\dH^{n-1}=0$. The quantities $B_{\Gamma}$
 and $B_{\partial\Omega}$ stand for the  second fundamental form of $\Gamma$ and $\partial\Omega$, respectively,  $\norm{B_{\Gamma}}^2$
denotes the norm squared of the second fundamental form--or equivalently, the sum of the squares of the $n-1$ principal curvatures of $\Gamma$,
and $G:\Omega\times\Omega\to\R$ denotes the Green's function for $-\Delta$ in $\Omega$ subject to homogeneous
Neumann boundary conditions. The function $v$ in the last term above denotes the solution to the Poisson equation \eqref{Poisson}
and $n_{\Gamma}$ denotes the outer unit normal with respect to $\{u=1\}$. We refer
to \cite{CS} or \cite{CM} for details.

Let us now apply \eqref{gensecondvar} to the setting of the previous sections by taking $\Omega=\Oe=(0,\e)\times (0,1)$ and $u=u_k$ given by \eqref{kodd}-\eqref{keven} for
any positive integer $k$.
Denoting by $\Gamma$ the union $\cup_{j=1}^k \Gamma_j$ of line segments $\Gamma_j:= (0,\e)\times\{y_j\}$
comprising the jump set of $u_k$, we evaluate $\delta^2 \E(u_k;f)$ for an arbitrary function $f\in H^1(\Gamma)$ satisfying
\beq
\int_{\Gamma}f\,\dH^1=0\label{zeroint}
\eeq
to find
\begin{eqnarray}
&&\delta^2 \E(u_k;f)=\nonumber\\
&&\sum_{j=1}^k\int_0^\e f_j'(x)^2\,dx+8\gamma\int_{\Gamma}\int_{\Gamma} G\,f\,f\,\dH^1\,\dH^1+
4\gamma \int_{\Gamma}\bigg(\nabla v_k\cdot n_{\Gamma}\bigg)\,f^2\,\dH^1.\label{sv}
\end{eqnarray}
Here we have used the fact that ${B_{\Gamma}}\equiv 0$ and $B_{\partial\Omega}=0$ in a neighborhood of  $\overline \Gamma\cap\partial\Omega$, 
 and we have introduced $f_j$ for the restriction $f\rtangle\Gamma_j$. It will also be convenient to introduce
the notation $f_j=a_j+g_j$ where $a_j:=\fint f_j$ so that by \eqref{zeroint} we have
\beq
\sum_{j=1}^k a_j=0\quad\mbox{and for each}\;j\;\mbox{we have}\quad
\int_0^\e g_j(x)\,dx=0.\label{zero}
\eeq
We will analyze each term of \eqref{sv} separately. Starting with the first one, we note that by the Poincar\'e inequality,
one has
\beq
\int_0^\e f_j'(x)^2\,dx=\int_0^\e g_j'(x)^2\,dx\geq \bigg(\frac{\pi}{\e}\bigg)^2\int_0^\e g_j(x)^2\,dx.\label{Poincare}
\eeq
Next, to analyze the term involving the Green's function we need a bit more notation. Let $\bar{f}:\Gamma\to\R$ denote the
function given by $\bar{f}\rtangle\Gamma_j=a_j$ and let $g:\Gamma\to\R$ be the function given by $g\rtangle\Gamma_j=g_j$.
Also, we introduce the measures $\mu_{\bar{f}}$ and $\mu_g$ via the formulas
\[
\mu_{\bar{f}}=\sum_{j=1}^k a_j\delta_{\Gamma_j},\qquad \mu_g=\sum_{j=1}^k g_j\delta_{\Gamma_j}
\]
and let $v_{\bar{f}}$ and $v_g$ denote the weak $H^1$ solutions to the Poisson equations
\[
-\Delta v_{\bar{f}}= \mu_{\bar{f}},\qquad -\Delta v_g=\mu_g\quad\mbox{in}\;\Oe
\]
subject to homogeneous Neumann boundary conditions and zero mean.

Note that $v_{\bar{f}}$ will depend only on $y$ so that
\begin{eqnarray}
&&\int_{\Gamma}\int_{\Gamma} G\,\bar{f}\,g\,\dH^1\,\dH^1=\int_{\Oe}\int_{\Oe}G\,d\mu_{\bar{f}}\,d\mu_g=\int_{\Oe}v_{\bar{f}}\,d\mu_g\nonumber\\
&&=\sum_{j=1}^k\int_0^\e\bigg(v_{\bar{f}}\rtangle \Gamma_j\bigg)\,g_j(x)\,dx=
\sum_{j=1}^k\bigg(v_{\bar{f}}\rtangle \Gamma_j\bigg)\,\int_0^\e g_j(x)\,dx=0\label{nocross}
\end{eqnarray}
by \eqref{zero}. Therefore, we find that
\begin{eqnarray*}
&&\int_{\Gamma}\int_{\Gamma} G\,f\,f\,\dH^1\,\dH^1=\int_{\Gamma}\int_{\Gamma} G\,(\bar{f}+g)(\bar{f}+g)\,\dH^1\,\dH^1=\int_{\Gamma}\int_{\Gamma} G\,\bar{f}\,\bar{f}\,\dH^1\,\dH^1+\int_{\Gamma}\int_{\Gamma} G\,g\,g\,\dH^1\,\dH^1\\
&&=\int_{\Oe}\int_{\Oe}G\,d\mu_{\bar{f}}\,d\mu_{\bar{f}}+\int_{\Oe}\int_{\Oe}G\,d\mu_g\,d\mu_g=
\int_{\Oe}\abs{\nabla v_{\bar{f}}}^2\,dx\,dy+\int_{\Oe}\abs{\nabla v_{g}}^2\,dx\,dy.
\end{eqnarray*}
Now since $v_{\bar{f}}$ satisfies $-v_{\bar{f}}''=a_1\delta_{\{y=y_1\}}+\ldots+a_k\delta_{\{y=y_k\}}$ with
$v_{\bar{f}}'(0)=0=v_{\bar{f}}'(1)$, we can integrate and use \eqref{zero} to obtain
\[
v_{\bar{f}}'(y)=\left\{\begin{matrix} 0&\mbox{for}\;0<y<y_1\\
 -a_1&\mbox{for}\;y_1<y<y_2\\
  -(a_1+a_2)&\mbox{for}\;y_2<y<y_3\\
\cdot&\\
\cdot&\\
\cdot&\\
-(a_1+a_2+\ldots+a_{k-1})&\mbox{for}\;y_{k-1}<y<y_k\\
0&\mbox{for}\;y_k<y<1.\end{matrix}\right.
\]
This allows us to compute the value of $\int_{\Oe}\abs{\nabla v_{\bar{f}}}^2=\e\int_0^1(v_{\bar{f}}')^2(y)\,dy$
and we find
\beq
\int_{\Gamma}\int_{\Gamma} G\,f\,f=\frac{\e}{k}\left[a_1^2+(a_1+a_2)^2+\ldots+(a_1+a_2+\ldots+a_{k-1})^2\right]
+\int_{\Oe}\abs{\nabla v_{g}}^2.\label{Green}
\eeq
It remains to compute the last integral in \eqref{sv}.  In view of \eqref{1dode}, \eqref{kodd} and \eqref{keven} we have
the alternating pattern
\[
v_k'(y_1)=-\frac{1}{2k},\;v_k'(y_2)=\frac{1}{2k},\;v_k'(y_3)=-\frac{1}{2k},\ldots
\]
(cf. Figure \ref{vk}).
Recalling that $\nabla v_k\cdot n$ denotes the outer normal derivative
with respect to the set $\{u_k=1\}$, it then follows from \eqref{zero} that
\begin{eqnarray}
&&
\int_{\Gamma}\bigg(\nabla v_k\cdot n\bigg)f^2\dH^1=\sum_{j=1}^k\int_{\Gamma_j}\bigg(\nabla v_k\cdot n\bigg)\bigg(a_j+g_j(x)\bigg)^2\,dx\nonumber\\
&&=v_k'(y_1)\int_{\Gamma_1}\bigg(a_1+g_1(x)\bigg)^2\,dx-v_k'(y_2)\int_{\Gamma_2}\bigg(a_2+g_2(x)\bigg)^2\,dx+\ldots-(-1)^k
v_k'(y_k)\int_{\Gamma_k}\bigg(a_k+g_k(x)\bigg)^2\,dx\nonumber\\
&&=\;-\frac{\e}{2k}\sum_{j=1}^k a_j^2\;-\frac{1}{2k}\sum_{j=1}^k\int_0^\e g_j(x)^2\,dx.\label{flux}
\end{eqnarray}
Combining \eqref{Poincare}, \eqref{Green} and \eqref{flux} we conclude that
\begin{eqnarray}
&&
\delta^2 \E(u_k;f)\geq
\left(\bigg(\frac{\pi}{\e}\bigg)^2-\frac{2\gamma}{k}\right)\sum_{j=1}^k\int_0^\e g_j(x)^2\,dx+
8\gamma \int_{\Oe}\abs{\nabla v_{g}}^2\nonumber\\
&& +\bigg(\frac{2\gamma\e}{k}\bigg)\bigg(4\left[a_1^2+(a_1+a_2)^2+\ldots+(a_1+a_2+\ldots+a_{k-1})^2\right]-
\left[a_1^2+a_2^2+\ldots+a_k^2\right]\bigg).\nonumber\\
&&\label{2sv}
\end{eqnarray}
We now claim that the quadratic form arising in the last line of \eqref{2sv} is positive definite. To see this, it is
convenient to change variables in this expression by introducing
\[
\alpha_1=a_1,\;\alpha_2=a_1+a_2,\;\ldots,\;\alpha_k=a_1+a_2+\ldots+a_k.\]
Then rewriting the expression in terms of the $\alpha_j's$ and using \eqref{zero} we find after a little algebra that
\begin{eqnarray*}
&&
4\left[a_1^2+(a_1+a_2)^2+\ldots+(a_1+a_2+\ldots+a_{k-1})^2\right]-
\left[a_1^2+a_2^2+\ldots+a_k^2\right]\\
&&=4\left[\alpha_1^2+\alpha_2^2+\ldots+\alpha_{k-1}^2\right]-
\left[\alpha_1^2+(\alpha_2-\alpha_1)^2+\ldots+(\alpha_{k-1}-\alpha_{k-2})^2+\alpha_{k-1}^2\right]\\
&&
=2\bigg(\alpha_1^2+\alpha_2^2+\ldots+\alpha_{k-1}^2+\alpha_1\alpha_2+\alpha_2\alpha_3+\ldots+\alpha_{k-2}\alpha_{k-1}\bigg)
\end{eqnarray*}
\[=(\alpha_1,\alpha_2,\ldots,\alpha_{k-1})\left(\begin{array}{cccc}2&1&&0\\
1&\ddots&\ddots&\\
&\ddots&\ddots&1\\0 & & 1 & 2\\\end{array}\right)\left(\begin{matrix}\alpha_1\\ \alpha_2\\
\cdot \\ \cdot \\ \cdot\\ \alpha_{k-1}\end{matrix}\right)\]\[\geq
\bigg\{2+2\cos{\bigg(\frac{(k-1)\pi}{k}\bigg)}\bigg\}\left[\alpha_1^2+\alpha_2^2+\ldots+\alpha_{k-1}^2\right]
\]
since the eigenvalues of this matrix are given by
\[
\lambda_j=2+2\cos{\bigg(\frac{j\pi}{k}\bigg)}\quad\mbox{for}\;j=1,2,\ldots,k-1
\]
(see e.g. \cite{LL}).

In particular, returning to \eqref{2sv} we have established
\begin{prop}\label{Stability}
For any positive integer $k$, the function $u_k$ given by \eqref{kodd}-\eqref{keven}
is a stable critical point of the functional $\E$ provided
\beq
\e<\pi\sqrt{\frac{k}{2\gamma}}.\label{stable}
\eeq
\end{prop}
We note that the stability of the lamellar configurations implies that they are in fact $L^1$ 
local minimizers, as made precise by Theorem \ref{local} to follow.
\section{Two-dimensional minimality of one-dimensional\\ minimizers }

In this section we prove our main result. A crucial tool will be the the recent work in \cite{AFM} and \cite{Julin} on stability implying local minimality
for the nonlocal isoperimetric problem. Here we present an adaptation applicable to the present setting of cylindrical domains
with Neumann boundary conditions. 
\begin{thm}\label{local}Let $\Omega\subset\R^n$ be any bounded smooth domain with $n$ arbitrary and for any positive integer
$\ell$ and any positive $\e$ let $\Omega^{\ell}_{\e}\subset\R^{\ell+n}$ be defined as $(0,\e)^{\ell}\times\Omega$. Then for any $\gamma\geq 0$, given any {\em regular critical point}
$u\in L^1(\Omega^{\ell}_{\e})$ of $E^{\gamma}_{\Omega^{\ell}_{\e}}$ such that $\overline{\partial\{u=1\}\cap\Omega^{\ell}_{\e}}$ 
only meets the regular part of $\partial\Omega^{\ell}_\e$ and  $\delta^2 E^{\gamma}_{\Omega^{\ell}_\e}(u;f)>0$ for all nontrivial
$f\in H^1(\partial\{u=1\}\cap\Omega^{\ell}_\e)$ satisfying $\int f=0$ , there exist
$\delta$ and $C>0$ such that
\[
E^{\gamma}_{\Omega^{\ell}_{\e}}(w)\geq E^{\gamma}_{\Omega^{\ell}_{\e}}(u)+C\norm{u-w}_{L^1(\Omega^{\ell}_\e)}^2\quad
\mbox{whenever}\quad \norm{u-w}_{L^1(\Omega^{\ell}_\e)}<\delta\text{ and }\int_\Omega w\, dx=\int_\Omega u\, dx\,. \]
\end{thm}
As mentioned before, the proof of the theorem is essentially contained in \cite{AFM} and \cite{Julin}, but a few remarks are in order. Let us recall the classical notion of quasiminimizers  of the standard perimeter.
We say that a set $E\subset\Omega$ of (relative) locally finite perimeter is a {\em strong quasiminimizer in $\Omega$} with constants $\Lambda>0$ and $r>0$ if for every $F\subset \Omega$ of (relative) locally finite perimeter with $F\Delta E\subset B_r(x_0)$ for some ball $B_r(x_0)$ we have that  
$$
P_{\Omega\cap B_r(x_0)}(E)\leq P_{\Omega\cap B_r(x_0)}(F)+\Lambda|E\Delta F|\,.
$$
We recall that in particular, local minimizers of the nonlocal isoperimetric problem are quasiminimizers, cf. e.g.
\cite[Theorem 2.8]{AFM}.
The following theorem follows from the well-established regularity theory for quasiminimizers of the perimeter (see for instance \cite[Theorem 3.3]{Julin}).
\begin{thm}\label{th:almgrenbis}
Assume that $\Omega$ is smooth and let $E_h\subset\Omega$ be  a sequence of  strong quasiminimizers  in $\Omega$ with uniform constants $\Lambda>0$ and $r>0$ and such that
$$
\chi_{E_h}\to \chi_E\quad\text{a.e. in  $\Omega$ as $h\to\infty$}
$$
for some set $E$   of class $C^{1, \alpha}$, $\alpha\in (0, \frac12)$, such that 
either  $\overline {\partial E\cap\Omega}\cap \partial \Omega=\emptyset$ or  
$\overline {\partial E\cap\Omega}$ meets $\partial \Omega$ orthogonally.
Then,  for $h$ large enough $\partial E_h$ is of class $C^{1, \alpha}$  and 
\begin{equation}\label{nonlocal}
\partial E_h\to \partial E\qquad\text{in $C^{1,\alpha}$.}
\end{equation}
\end{thm}
The convergence in \eqref{nonlocal} can be restated equivalently by saying that we may find a sequence $\Phi_h$ of diffeomorphisms of class $C^{1,\alpha}$ from $\overline \Omega$ onto itself  such that $\Phi_h(\partial E)=\partial E_h$ and $\|\Phi_h-I\|_{C^{1,\alpha}}\to 0$, where $I$ denotes the identity map. 

The following corollary is an adaptation of Theorem~\ref{th:almgrenbis} to the case of the cylindrical domains $\Omega^{\ell}_\e$ considered in Theorem~\ref{local}.
To state it, we need to introduce some notation for even extension of sets in $\Omega^{\ell}_\e$. Let us express any $z\in (0,\e)^{\ell}\times\Omega$ as 
$z=(x_1,\dots, x_\ell, y_1, \dots, y_n)$, with 
$(x_1,\dots, x_\ell)\in (0,\e)^{\ell}$ and $(y_1, \dots, y_n)\in \Omega$. Then given any set $E\subset  \Omega^{\ell}_\e$ we may perform infinitely many  even reflections of the characteristic function $\chi_{E}(x_1,\dots, x_\ell, y_1, \dots, y_n)$ with respect to the $x_1,\dots, x_\ell$ variables, to obtain the characteristic function of a set  we denote by $\tilde E\subset \R^{\ell}\times \Omega$. 
\begin{cor}\label{cor:almgren}
Let $E_h\subset\Omega^{\ell}_\e$ be  a sequence of  strong quasiminimizers in $\Omega^{\ell}_\e$ with uniform constants $\Lambda>0$ and $r>0$ and such that
$$
\chi_{E_h}\to \chi_E\quad\text{a.e.  in $\Omega^{\ell}_\e$}
$$
for some set $E\subset \Omega^{\ell}_\e$ such that $\tilde{E}$ is of class $C^{1, \alpha}(\R^{\ell}\times\Omega)$, 
$\alpha\in (0, \frac12)$, and  either  $\overline {\partial \tilde E\cap(\mathbb{R}^\ell\times 
\Omega)}\cap \partial (\mathbb{R}^\ell\times \Omega)=\emptyset$ or $\overline {\partial \tilde 
E\cap(\mathbb{R}^\ell\times \Omega)}$  meets $\partial (\mathbb{R}^\ell\times \Omega)$ 
orthogonally.

  Then,  for $h$ large enough $\partial E_h$ is of class $C^{1, \alpha}$  and $\{E_h\}$ satisfies \eqref{nonlocal}.
\end{cor}
\begin{proof}
The point here is that the boundary $\partial \Omega^{\ell}_\e$ has a ``singular'' part. 
The trick is to remove these singularities by reflection. 
 It is straightforward to check that the reflected sets $\tilde E_h$ are strong quasiminimizers with the same uniform constants $\Lambda>0$ and $r>0$, and that
$$
\chi_{\tilde E_h}\to \chi_{\tilde E}\quad\text{a.e. in $ \R^{\ell}\times \Omega$.}
$$
The conclusion then follows by applying Theorem~\ref{th:almgrenbis} with $\Omega$ replaced by $\R^{\ell}\times \Omega$.

\end{proof}
\noindent
{\bf Proof of Theorem \ref{local}.} 
As mentioned before the proof is essentially contained in \cite{Julin}, where the general strategy devised in \cite{AFM} in the periodic setting has been adapted to the Neumann case. It consists of two main steps.

\noindent{\bf Step 1.} One shows that the positive definiteness of $\delta^2 E^{\gamma}_{\Omega^{\ell}_\e}(u;f)$ implies that 
$u$ is an isolated local minimizer with respect to small $W^{2,p}$- perturbations, for all $p$ sufficiently large, of the free-boundary 
$\partial\{u=1\}\cap\Omega^{\ell}_\e$. Precisely, for all $p$ sufficiently large one can show   the existence of $\delta>0$ such that if 
$\Phi:\overline\Omega^{\ell}_\e\to \overline\Omega^{\ell}_\e$ is a diffeomorphism of class $W^{2,p}$ and $\int_{\Omega^{\ell}_{\e}}u\circ\Phi\, dx=\int_{\Omega^{\ell}_{\e}}u\, dx$, then 
\[
E^{\gamma}_{\Omega^{\ell}_{\e}}(u\circ\Phi)\geq E^{\gamma}_{\Omega^{\ell}_{\e}}(u)+C\norm{u-(u\circ\Phi)}_{L^1(\Omega^{\ell}_\e)}^2
\quad\mbox{provided}\quad \norm{\Phi-I}_{W^{2,p}(\Omega^{\ell}_\e)}<\delta.\]
This fact follows from \cite[Proposition 5.2]{Julin}: indeed,  due to the assumptions on $\partial\{u=1\}\cap\Omega^{\ell}_\e$,  
the argument is not affected by the presence of a ``singular'' part in $\partial\Omega^{\ell}_\e$.

\noindent{\bf Step 2.} One shows that the conclusion of the previous step implies  the thesis of the Theorem. 
This can be argued exactly as in \cite[Section 6]{Julin} (see also \cite[Proof of Theorem 1.1]{AFM}). The proof can be reproduced word for word, using  Corollary~\ref{cor:almgren} instead of  \cite[Theorem 3.3]{Julin}.
\qed

 As in the previous two sections, for convenience only at this point we fix
the mass constraint $m$ to be zero. We now present our main result for the nonlocal isoperimetric problem
posed on the domain $\Omega_\e:=(0,\e)\times (0,1)$. (This corresponds to $\ell=1$ and $\Omega= (0,1)$ in the notation $\Omega^{\ell}_\e$ used previously in this section.) 
\begin{thm}\label{main} For any positive integer $k$, fix $\gamma>0$ in the interval given by \eqref{kwins}. 
Let $a$ be any positive number smaller than $\pi\sqrt{\frac{k}{2\gamma}}$.
Then for all sufficiently large integers $j$, the minimizers $\pm u_k$ given by \eqref{kodd}-\eqref{keven}
of the one-dimensional energy $\Ez$
are also the minimizers of the two-dimensional energy $E^{\gamma}_{\Omega_{\e_j}}$ where $\e_j:=\frac{a}{j}$
and are the only minimizers of this energy.
 \end{thm}
 
 A typical minimizer is depicted in Figure \ref{stripes}.
The proof consists of a combination of the $\Gamma$-convergence of Section 2, the one-dimensional minimality of $u_k$ established
in Section 3, the two-dimensional stability shown in the previous section and Theorem \ref{local}. 

\begin{figure}
\centerline{{\includegraphics[scale = 0.5, clip = true, trim = 5cm 12cm 8cm 6cm]{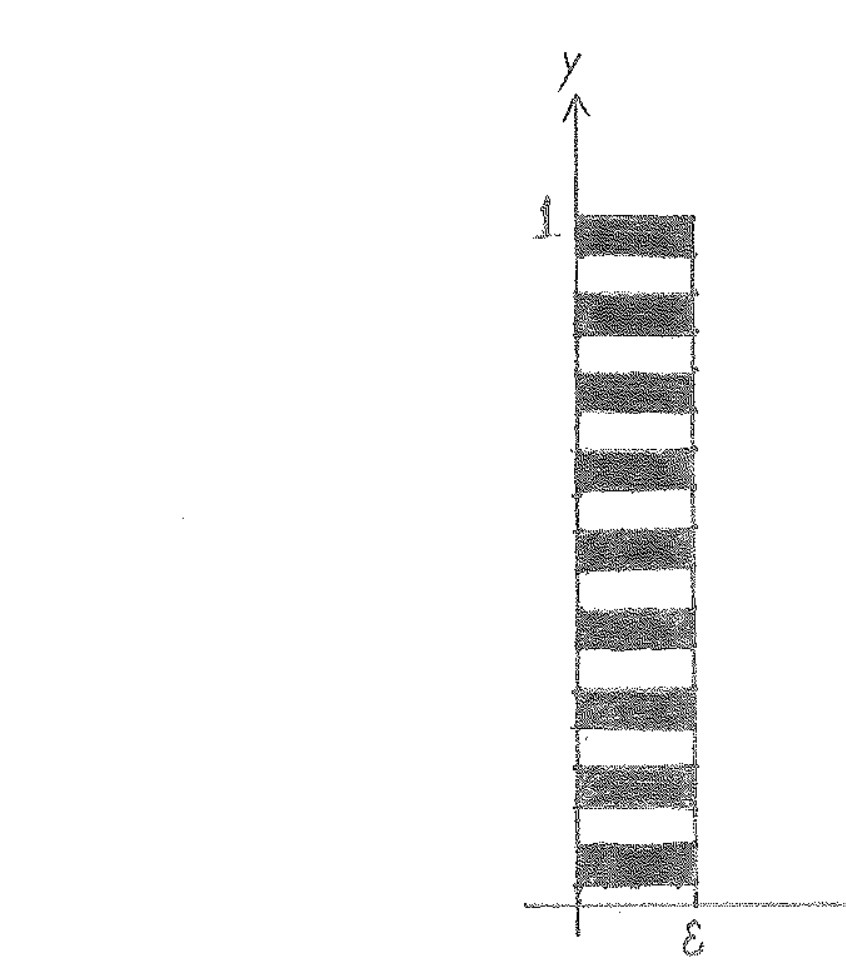}}} \caption{Graph of
a typical minimizer of $E^{\gamma}_{\Omega_\e}$.}
\label{stripes}
\end{figure}

\noindent
{\bf Proof of Theorem \ref{main}.} Throughout the proof, $k$ and then $\gamma$ are fixed so that $u_k$ minimizes $\Ez.$
For any positive integer $j$, let us denote by $u_{\e_j}$ a global minimizer of $E^{\gamma}_{\Omega_{\e_j}}.$ We will argue that
$u_{\e_j}=u_k$ for all large enough integers $j$. 

Note first that from the choice of $a$,
Proposition \ref{Stability} guarantees that $u_k$ is a stable critical point of $E^{\gamma}_{\Omega_a}$.
Applying Theorem \ref{local}, we can then assert the local minimality of $u_k$ in $\Omega_a$, namely the existence of a positive $\delta$ and $C$ such that
\begin{equation}
E^{\gamma}_{\Omega_a}(w)\geq E^{\gamma}_{\Omega_a}(u_k)+C\norm{w-u_k}_{L^1(\Omega_a)}^2\quad\mbox{provided}\quad \norm{w-u_k}_{L^1(\Omega_a)}<\delta.
\label{stab2}\end{equation}

With $a$ now fixed, we apply the $\Gamma$-convergence result Theorem \ref{gc} to the sequence of
functionals $\tilde{E}^{\gamma}_j:L^1(\Omega_a)\to \R$ defined via
\[
\tilde{E}^{\gamma}_j(\tilde{u}):=\frac{1}{\e_j}E^{\gamma}_{\Omega_{\e_j}}(u)\quad
\mbox{where}\quad u(x,y):=\tilde{u}(j\,x,y)\quad
\mbox{for any}\quad \tilde{u}\in L^1(\Omega_a),
\]
 cf. \eqref{tEdefn}. Of course, Theorem \ref{gc} is phrased
in terms of a sequence of rescaled nonlocal isoperimetric problems defined on the unit square $\Omega_1$, corresponding to $a=1$ in the present notation,
but the result is unchanged
if we replace $\Omega_1$ by $\Omega_a$. Since convergent sequences of minimizers have a limit which minimizes the $\Gamma$-limit $E_0^{\gamma}$, we conclude that the
sequence $\tilde{u}_{j}:\Omega_a\to\R$ given by $\tilde{u}_{j}(x,y):=u_{\e_j}(x/j,y)$ which minimizes
$\tilde{E}^{\gamma}_j$ must satisfy the condition
\begin{equation}
\tilde{u}_j\to u_k\;\mbox{or}\;-u_k\;\mbox{in}\; L^1(\Omega_a)\;\mbox{as}\;j\to\infty.\label{ggc}
\end{equation}
The indeterminacy in \eqref{ggc} is simply due to the nonuniqueness associated with the choice of mass constraint
$m=0$ since for any $u$ one has $\tilde{E}^{\gamma}_j(u)=\tilde{E}^{\gamma}_j(-u)$ and $E_0^{\gamma}(u)=E_0^{\gamma}(-u)$.
Let us adopt the convention that if necessary, we multiply $\tilde{u}_j$ by $-1$ so as to obtain $\tilde{u}_j\to u_k$ in $L^1$.

Now for any integer $j>1$, we evenly reflect $j-1$ times with respect to $x$ the minimizer $u_{\e_j}:\Omega_{\e_j}\to\R$  to build
a function defined in $\Omega_a$ that we denote by $u^r_{\e_j}$. If we then denote by $v_{\e_j}$ the solution to the Poisson
problem \eqref{Poisson} in $\Omega_{\e_j}$ with right-hand side $u_{\e_j}$, one readily checks that the solution to \eqref{Poisson}
in $\Omega_a$ with right-hand side $u^r_{\e_j}$ is simply given by the repeated even reflection of $v_{\e_j}$ as well. Note in particular
that even reflection preserves the required homogeneous Neumann boundary conditions. We write
$v^r_{\e_j}$ for this reflection of $v_{\e_j}$.

We next observe that
\[
\norm{u^r_{\e_j}-u_k}_{L^1(\Omega_a)}=j\,\norm{u_{\e_j}-u_k}_{L^1(\Omega_{\e_j})}
=\int_0^1\int_0^{a} \abs{u_{\e_j}(x/j,y)-u_k}\,dx\,dy=\norm{\tilde{u}_j-u_k}_{L^1(\Omega_a)}.
\]
Invoking \eqref{ggc}, we conclude that
$\norm{u^r_{\e_j}-u_k}_{L^1(\Omega_a)}<\delta$ for all sufficiently large integers $j$.
Consequently we may apply \eqref{stab2} with $w=u^r_{\e_j}$ to find that
\begin{equation}
E^{\gamma}_{\Omega_a}(u^r_{\e_j})\geq E^{\gamma}_{\Omega_a}(u_k)+C\norm{u^r_{\e_j}-u_k}_{L^1(\Omega_a)}^2.
\label{end}\end{equation}

However, for both $u^r_{\e_j}$ and $u_k$ the contribution to the energy $E^{\gamma}_{\Omega_a}$ within
each of the $j$ rectangles of width $\e_j$ is identical, so that 
\[E^{\gamma}_{\Omega_a}(u^r_{\e_j})=j\,E^{\gamma}_{\Omega_{\e_j}}(u_{\e_j})\quad
\mbox{and}\quad E^{\gamma}_{\Omega_a}(u_k)=j\,E^{\gamma}_{\Omega_{\e_j}}(u_k).\]
 By \eqref{end}, it follows that
\[
E^{\gamma}_{\Omega_{\e_j}}(u_{\e_j})\geq E^{\gamma}_{\Omega_{\e_j}}(u_k)+\,C\,\norm{u_{\e_j}-u_k}_{L^1(\Omega_{\e_j})}^2\quad
\mbox{for all sufficiently large integers}\;j,\]
contradicting the minimality of $u_{\e_j}$ in $\Omega_{\e_j}$ unless $u_{\e_j}\equiv u_k$.\qed
\brk
The reason for the restriction in the statement of Theorem \ref{main} to rectangles
of width $\frac{a}{j}$ is to allow for use of the reflection argument in the proof.
One could remove this restriction by strengthening Theorem \ref{local}, 
specifically by showing that under the same assumptions the conclusion holds not only for $\Omega_\e$ but 
also for $\Omega_\eta$ for all $\eta$ sufficiently close to $\e$ and with $\delta$ and $C$ 
independent of $\eta$. This could no doubt be accomplished by repeating the argument of \cite{AFM} or \cite{Julin}, and by verifying that the minimality neighborhood is 
independent of $\eta$. However, we have not checked the details.
\erk

\section{Generalizations to higher dimensions}

We conclude with two generalizations of Theorem \ref{main} applicable in higher dimensions
to indicate the scope of the method. In the first, we consider $\E$ on a thin rectangular box in arbitrary dimension
that collapses with $\e$ to a line segment.
Then one can again assert that global minimizers of the one-dimensional problem remain global minimizers on
sufficiently thin boxes:
\begin{thm}\label{main2} Let $k$ and $\ell$ be any positive integers and fix $\gamma>0$ in the interval given by \eqref{kwins}. 
Let $a$ be any positive number less than  $\pi\sqrt{\frac{k}{2\gamma}}$.
Then for all sufficiently large integers $j$, the minimizers $\pm u_k$ given by \eqref{kodd}-\eqref{keven}
of the one-dimensional energy $\Ez$
are also minimizers of the $(\ell+1)$-dimensional energy $E^{\gamma}_{\Omega^{\ell}_{\e_j}}$ given by \eqref{Egammaintro} posed on the domain $\Omega^{\ell}_{\e_j}=
(0,\e_j)^{\ell}\times (0,1)$
 where $\e_j:=\frac{a}{j}$. Furthermore, these are the only minimizers of this energy.
 \end{thm}
The proof of the theorem needs the following adaptation of Theorem~\ref{local} to the present setting.
\begin{thm}\label{localbis} Assume that for a positive integer $k$  the lamellar configuration $u_k$ given by \eqref{kodd}-\eqref{keven} satisfies $\delta^2 E^{\gamma}_{\Oe}(u_k;f)>0$ for all nontrivial
$f\in H^1(\partial\{u_k=1\}\cap\Oe)$ satisfying $\int f=0$.
 Then,  $u_k$  is an isolated local $L^1$-minimizer; i.e., 
there exist $\delta$ and $C>0$ such that
\[
E^{\gamma}_{\Omega_{\e}}(w)\geq E^{\gamma}_{\Omega_{\e}}(u_k)+C\norm{u_k-w}_{L^1(\Oe)}^2\quad\mbox{whenever}\quad \norm{u_k-w}_{L^1(\Oe)}<\delta\text{ and }\int_{\Oe} w\, dx=\int_{\Oe} u_k\, dx.\]
 \end{thm}
\noindent {\bf Proof of Theorem \ref{localbis}}.
Unlike the situation in Theorem~\ref{local}, here $\overline{\partial\{u_k=1\}\cap\Oe}$  meets also the 
non-regular part of $\partial\Oe$. However, we can take advantage of the fact that we deal with a particular configuration, having flat interfaces. We denote by $\Gamma_1$,\dots, $\Gamma_k$ the $k$ flat interfaces of $u_k$.   As in the first step of the proof for Theorem~\ref{local}, 
 one starts by deducing the isolated local minimality of $u_k$ with respect to configurations whose interfaces are small $W^{2,p}$-perturbations for $p$ sufficiently large,  of $\Gamma_1$,\dots, $\Gamma_k$. To show this, one may assume that such interfaces are described by the graphs of functions $h_i\in W^{2,p}(\Gamma_i)$, $i=1,\dots, k$, with 
 $\sum_i\int_{\Gamma_i}h_i=0$, $\nabla h_i\cdot n_{\Gamma_i}=0$ on $\partial\Gamma_i$ away from corners and $\sum_i\|h_i\|_{W^{2,p}(\Gamma_i)}$ small enough. Thus, it is possible construct a volume preserving flow $\Phi$ connecting $u_k$ with the perturbed configurations  such that
 $\Phi(t,x)=x+th_i(\pi_i(x))$ in a neighborhood of $\Gamma_i$ for all $i=1,\dots, k$, where $\pi_i$ denotes the orthogonal projection on $\Gamma_i$. Now, one can argue exactly as in \cite[Proposition 5.2]{Julin}, using such a flow instead of the one constructed in \cite[Lemma 5.3]{Julin} (see also \cite[Theorem 6.2]{AFM}). 
Once the $W^{2,p}$-minimality is established,  the conclusion of the theorem follows exactly as in the second step of the proof of Theorem~\ref{local} (through an appeal to Corollary~\ref{cor:almgren}).
\qed

\noindent
{\bf Proof of Theorem \ref{main2}.} For $\ell=1$ this result reduces to Theorem \ref{main}.  After rescaling the problem onto the unit cube in $\ell+1$ dimensions, the identity \eqref{tEdefn}
is replaced by
\begin{eqnarray*}&&\frac{1}{\e^{\ell}}E^{\gamma}_{\Omega^{\ell}_{\e}}(u)=\tilde{E}^{\gamma}_{\e,\ell}(\tilde{u}):=\\&& \int_{\partial^*\{\tilde{u}=1\}\cap\Omega_1}\sqrt{\frac{1}{\e^2}(n_1^2+\ldots
+n_{\ell}^2)+n_{\ell+1}^2}\,\dH^{\ell}
+\gamma\int_{\Omega_1}\left(\frac{1}{\e^2}(\tilde{v}_{x_1}^2+\ldots+\tilde{v}_{x_{\ell}}^2)+\tilde{v}_{y_1}^2\right)\,dx_1\,\dots\,dx_{\ell}\,dy_1
\end{eqnarray*}
where for any $u\in BV(\Oe;\{\pm 1\})$, we now
denote by $\tilde{u}:\Omega_1\to\R$ the function satisfying $\tilde{u}(x_1,\ldots,x_{\ell},y_1)=u(\e x_1,\ldots,\e x_{\ell},y_1)$ with a similar
definition relating the original potential $v$ associated with $u$ to the rescaled one $\tilde{v}.$  With this modification, the proof
of the $\Gamma$-convergence result Theorem \ref{gc} proceeds without change.

Regarding the stability of the one-dimensional minimizer in higher dimensions, the statement and proof of
Proposition \ref{Stability} are unchanged in the setting where we replace $\Omega_\e$ by $\Omega^{\ell}_\e.$
Thus we again have that $u_k$ is stable with respect to $E^{\gamma}_{\Omega^{\ell}_{\e}}$ provided $\e<\pi\sqrt{\frac{k}{2\gamma}}$.
The proof of Theorem \ref{main2} then proceeds
as in the proof of Theorem \ref{main}, with the obvious alteration that the reflection of the minimizer $u_{\e_j}$ is carried out with respect
to all the `thin' variables $x_1,x_2,\ldots,x_{\ell}$, and using Theorem~\ref{localbis} instead of Theorem~\ref{local}. 
\qed

More generally, one can take $\Omega\subset\R^n$, $n\geq 2$ to be an arbitrary bounded smooth domain and consider the energy $E^{\gamma}_{\Omega^{\ell}_{\e}}$
in the setting where $\Omega^{\ell}_{\e}\subset\R^{n+l}$ is given by $\Omega^{\ell}_{\e}:= (0,\e)^{\ell}\times\Omega$ for some positive integer $\ell$. Then one can establish the following:

\begin{thm}\label{main3} Assume $\overline{u}:\Omega\to\R$ is the unique global minimizer of $E^{\gamma}_{\Omega}$ given by \eqref{Egammaintro}, with $\overline{\partial\{u=1\}\cap\Omega}\cap\partial\Omega=\emptyset$.  Assume $\overline{u}$ is regular in the sense described at the beginning of Section 3
and assume furthermore that $\overline{u}$ is stable in the sense of positivity of the second variation $\delta^2E^{\gamma}_{\Omega}(\overline{u};f)$ given by \eqref{gensecondvar}
for all nontrivial $f\in H^1(\Gamma)$ satisfying $\int_{\Gamma} f\,d\mathcal{H}^{n-1}=0$ with $\Gamma:=\partial\{\overline{u}=1\}\cap\Omega$.
Then there exists a positive number $a$ such that
for all sufficiently large integers $j$, $\overline{u}$ is also the unique minimizer of the $(n+\ell)$-dimensional energy $E^{\gamma}_{\Omega^{\ell}_{\e_j}}$
 where $\e_j:=\frac{a}{j}$.
 \end{thm}

 \noindent
 {\bf Proof.} The analog of the $\Gamma$-convergence result Theorem \ref{gc} requires no significant
 change in its proof and again Theorem \ref{local} still applies here. Regarding a stability result analogous
 to Proposition \ref{Stability}, of course one cannot expect such an explicit determination of the critical value of
 $\e$ below which one has stability of $\overline{u}$ with respect to $E^{\gamma}_{\Omega^{\ell}_\e}$. Instead we will argue that for $\e$ small enough, the stability
 of $\overline{u}$ in $\Omega^{\ell}_\e$ follows from its assumed stability in $\Omega.$

 For ease of presentation only, we will take $\ell=1$ and to simplify the notation we will write $\Omega_\e$ for $\Omega^1_\e$.
 Below we use $x$ to denote a variable in $(0,\e)$ and $y=(y_1,\ldots,y_n)$ to denote a variable on $\Gamma$ or in
 $\Omega$ as the context requires. Denoting $\Gamma_{\e}:=(0,\e)\times\Gamma$ we let $f_{\e}\in H^1(\Gamma_{\e})$
 be any nontrivial sequence satisfying the conditions
 \beq
 \int_{\Gamma_{\e}}f_{\e}\,dx\,d\mathcal{H}^{n-1}=0\quad \mbox{and}\quad\int_{\Gamma_{\e}}f_{\e}^2\,dx\,d\mathcal{H}^{n-1}=\e,\label{normal}
 \eeq
 where the last requirement is a convenient normalization. We wish to argue that $\delta^2 E^{\gamma}_{\Omega_\e}(\overline{u};f_{\e})>0$
 for all $\e$ small. 

Using that $\norm{B_{\Gamma_{\e}}}=\norm{B_{\Gamma}}$ we see that
\begin{eqnarray*}
&&
 \delta^2 E^{\gamma}_{\Omega_\e}(\overline{u};f_{\e})=\int_{\Gamma_{\e}} \left(\abs{\nabla_{\Gamma} f_{\e}}^2+
 (\frac{\partial f_{\e}}{\partial x})^2\right)\,dx\,d\mathcal{H}^{n-1}-
 \int_{\Gamma_{\e}} \norm{B_{\Gamma}}^2 f_{\e}^2\,dx\,d\mathcal{H}^{n-1}\\
 &&+8\gamma\int_{\Oe}\abs{\nabla v_{f_{\e}}}^2\,dx\,dy+4\gamma  \int_{\Gamma_{\e}} \nabla \overline{v}(y)\cdot n_{\Gamma}\,f_{\e}^2
\,dx\,d\mathcal{H}^{n-1}.
\end{eqnarray*}
Here $\overline{v}$ solves the Poisson equation with right-hand side $\overline{u}$ and
we adapt the approach of Section 4 in writing the Green's function term using the solution $v_{f_{\e}}$ to the Poisson equation
\beq
-\Delta v_{f_{\e}}=f_{\e}\,\delta_{\Gamma_{\e}}\quad\mbox{in}\;\Oe,\quad\nabla v_{f_{\e}}\cdot n_{\partial\Oe}=0\;\mbox{on}\;\partial\Oe.
\label{vP}
\eeq
Then changing variables by introducing $x_1=\frac{x}{\e}$ and setting $\tilde{f}_{\e}(x_1,y)=f_{\e}(\e x_1,y)$ and
$ \tilde{v}_{\tilde{f}_{\e}}(x_1,y)=v_{f_{\e}}(\e x_1,y)$
we easily calculate that
\begin{eqnarray}
&&
 \frac{1}{\e}\delta^2 E^{\gamma}_{\Omega_\e}(\overline{u};f_{\e})=\delta^2\tilde{E}^{\gamma}_{\e}(\overline{u};\tilde{f}_{\e}):=\nonumber\\ &&
 \int_{\Gamma_1} \left(\abs{\nabla_{\Gamma} \tilde{f}_{\e}}^2+\frac{1}{\e^2}
 (\frac{\partial \tilde{f}_{\e}}{\partial x_1})^2\right)\,dx_1\,d\mathcal{H}^{n-1}-
 \int_{\Gamma_1} \norm{B_{\Gamma}}^2 \tilde{f}_{\e}^2\,dx_1\,d\mathcal{H}^{n-1}\nonumber\\
 &&+8\gamma\int_{\Omega_1}\left(\abs{\nabla_y\, \tilde{v}_{\tilde{f}_{\e}}}^2+\frac{1}{\e^2}
 (\frac{\partial \tilde{v}_{\tilde{f}_{\e}}}{\partial x_1})^2\right)\,dx_1\,dy
 +4\gamma  \int_{\Gamma_1} \nabla \overline{v}(y)\cdot n_{\Gamma}\,\tilde{f}_{\e}^2
\,dx_1\,d\mathcal{H}^{n-1}.\label{secvar2}
 \end{eqnarray}
Also we note that \eqref{normal} transforms to the conditions
\beq
 \int_{\Gamma_{1}}\tilde{f}_{\e}\,d\mathcal{H}^n=0,\quad \mbox{and}\quad\int_{\Gamma_{1}}\tilde{f}_{\e}^2\,d\mathcal{H}^n=1.\label{normal2}
 \eeq

 Now we may assume $\liminf_{\e\to 0}\delta^2 \tilde{E}^{\gamma}_{\e}(\overline{u};\tilde{f}_{\e})<\infty$ or else we are done.
Such a bound implies uniform bounds along a subsequence $\{\tilde{f}_{\e_i}\}$ of the form
\[ \int_{\Gamma_1} \abs{\nabla_{\Gamma} \tilde{f}_{\e_i}}^2
\,dx_1\,d\mathcal{H}^{n-1}<C,\qquad   \int_{\Gamma_1}\big(\frac{\partial \tilde{f}_{\e_i}}{\partial x_1}\big)^2\,dx_1\,d\mathcal{H}^{n-1}<C\e_i^2.
\]
These bounds and \eqref{normal2} in turn lead to the following convergences along a further subsequence (with subsequential notation suppressed):
\begin{eqnarray}
&&\tilde{f}_{\e_i}\rightharpoonup f_0\;\mbox{in}\; H^1(\Gamma_1),\quad \tilde{f}_{\e_i}\to f_0\;\mbox{in}\; L^2(\Gamma_1)\quad\mbox{for some}\;f_0\in H^1(\Gamma_1)\label{con1}\\
&&\mbox{such that}\quad  \int_{\Gamma_{1}}f_0\,dx_1\,d\mathcal{H}^{n-1}=0,\quad\int_{\Gamma_{1}}f_0^2\,dx_1\,d\mathcal{H}^{n-1}=1,\quad\mbox{and}\quad f_0=f_0(y)\;\mbox{only}.\nonumber
\end{eqnarray}
Applying \eqref{con1} and the fact that $f_0$ is independent of $x_1$ to the first, second and fourth integrals in \eqref{secvar2} we conclude that
\begin{eqnarray}
&&\liminf_{\e_i\to 0}\int_{\Gamma_1} \left(\abs{\nabla_{\Gamma} \tilde{f}_{\e_i}}^2+\frac{1}{\e_i^2}
 (\frac{\partial \tilde{f}_{\e_i}}{\partial x_1})^2\right)\,dx_1\,d\mathcal{H}^{n-1}-
 \int_{\Gamma_1} \norm{B_{\Gamma}}^2 \tilde{f}_{\e_i}^2\,dx_1\,d\mathcal{H}^{n-1}\nonumber\\
&& +4\gamma  \int_{\Gamma_1} \nabla \overline{v}(y)\cdot n_{\Gamma}\,\tilde{f}_{\e_i}^2
\,dx_1\,d\mathcal{H}^{n-1}\nonumber\\
&&
\geq  \int_{\Gamma_1}\abs{\nabla_{\Gamma} f_0}^2\,dx_1\,d\mathcal{H}^{n-1}-
 \int_{\Gamma_1} \norm{B_{\Gamma}}^2 f_0^2\,dx_1\,d\mathcal{H}^{n-1}+4\gamma  \int_{\Gamma_1} \nabla \overline{v}(y)\cdot n_{\Gamma}\,f_0^2
\,dx_1\,d\mathcal{H}^{n-1}\nonumber\\
&&=\int_{\Gamma}\abs{\nabla_{\Gamma} f_0}^2\,d\mathcal{H}^{n-1}-
 \int_{\Gamma} \norm{B_{\Gamma}}^2 f_0^2\,d\mathcal{H}^{n-1}+4\gamma  \int_{\Gamma} \nabla \overline{v}(y)\cdot n_{\Gamma}\,f_0^2
\,d\mathcal{H}^{n-1}.\label{3terms}
\end{eqnarray}
It remains to handle the third integral of \eqref{secvar2}. To this end, we note that uniform bounds on
the sequence $\{\tilde{v}_{\tilde{f}_{\e_i}}\}$ follow as did the ones for $\{\tilde{f}_{\e_i}\}$ so that
(after passing to a subsequence) one finds
\beq
\tilde{v}_{\tilde{f}_{\e_i}}\rightharpoonup w\;\mbox{in}\; H^1(\Omega_1)\quad\mbox{and}\quad \tilde{v}_{\tilde{f}_{\e_i}}\to w\;\mbox{in}\; L^2(\Omega_1)\quad\mbox{for some}\;w\in H^1(\Omega_1)
\label{con3}\eeq
satisfying $w=w(y)$ only and $\int_{\Omega}w\,dy=0.$
It follows that for the third integral in \eqref{secvar2} one has
\beq
\liminf_{\e_i\to 0}\int_{\Omega_1}\left(\abs{\nabla_y\, \tilde{v}_{\tilde{f}_{\e_i}}}^2+\frac{1}{{\e_i}^2}
 (\frac{\partial \tilde{v}_{\tilde{f}_{\e_i}}}{\partial x_1})^2\right)\,dx_1\,dy\geq
 \int_{\Omega}\abs{\nabla w}^2\,dy.\label{lscend}
 \eeq
As a last step in establishing the stability of $\overline{u}$ in $\Oe$ we must identify $w$ as
the solution to the Poisson problem
\beq
-\Delta w=f_0\,\delta_{\Gamma},\qquad \nabla w\cdot n_{\partial \Omega}=0\;\mbox{on}\;\partial\Omega.
\label{widen}
\eeq
To see this we note from \eqref{vP} that $\tilde{v}_{\tilde{f}_{\e}}$ satisfies the equation
\[
-\Delta_y \tilde{v}_{\tilde{f}_{\e}}-\frac{1}{\e^2}\frac{\partial^2 \tilde{v}_{\tilde{f}_{\e}}}{\partial x_1^2}
=\tilde{f}_{\e}\,\delta_{\Gamma_1}\quad\mbox{in}\;\Omega_1,\quad\nabla \tilde{v}_{f_{\e}}\cdot n_{\partial\Omega_1}=0\;\mbox{on}\;\partial\Omega_1.
\]
Multiplication by a test function $\phi\in C^{\infty}\big(\overline{\Omega}\big)$ (independent of $x_1$) and integration by parts then leads to
\[
\int_{\Omega_1}\nabla_y  \tilde{v}_{f_{\e}}\cdot\nabla \phi\,dx_1\,dy=\int_{\Gamma_1}\tilde{f}_{\e}\phi\,dx_1\,dy.\]
Applying \eqref{con1} and \eqref{con3} and passing to the limit as $\e_i\to 0$  we obtain that $w$ indeed solves \eqref{widen}.
Combining \eqref{3terms} and \eqref{lscend} we conclude that
\[
 \liminf_{\e\to 0}\delta^2 \tilde{E}^{\gamma}_{\e}(\overline{u};\tilde{f}_{\e})\geq
\delta^2E^{\gamma}_{\Omega}(\overline{u};f_0)>0\]
by invoking the assumed stability of $\overline{u}$ as a critical point of $E^{\gamma}_{\Omega}$. Hence,
$\overline{u}$ remains stable in the thin domain $\Omega_{\e}$ as well.

The rest of the proof now follows as in the proof of Theorem \ref{main}, through an appeal to Theorem \ref{local} and the reflection
argument used before.\qed
\begin{rmrk}
The assumption $\overline{\partial\{u=1\}\cap\Omega}\cap\partial\Omega=\emptyset$ is certainly restrictive. It could be removed if one could extend Theorem~\ref{local} 
to the case where $\overline{\partial\{u=1\}\cap\Oe}$ also meets the non-regular part of $\partial\Oe$. Such an extension, which is very likely possible, would however require one to modify
some of the arguments presented in \cite{Julin}. As this goes beyond the purposes of the present paper,  we decided to state the previous theorem under more restrictive assumptions just to illustrate the  scope of the method.
\end{rmrk}
\vskip.2in
\noindent
{\bf Acknowledgments.} The research of M.M. was supported by the ERC grant 207573 ``Vectorial Problems."
The research of P.S.was supported by NSF grant DMS-1101290.

\end{document}